\documentclass[12pt]{amsart}
\usepackage{amssymb,amsmath,amsthm,mathrsfs,multirow,xcolor,framed,url}
\usepackage[pdftex,
         pdfauthor={Dandan Chen, Rong Chen and Siyu Yin},
         pdftitle={On the total number of ones associated with cranks of partitions modulo 11},
         pdfsubject={MATHEMATICS},
         pdfkeywords={Total number of ones, Crank for partitions, Theta-functions.},
         pdfproducer={Latex with hyperref},
         pdfcreator={pdflatex}]{hyperref}
\newtheorem{theorem}{Theorem}[section]
\newtheorem{lemma}[theorem]{Lemma}
\newtheorem{cor}[theorem]{Corollary}

\theoremstyle{definition}

\theoremstyle{remark}

\numberwithin{equation}{section}

\newcommand\nutwid{\overset {\text{\lower 3pt\hbox{$\sim$}}}\nu}

\allowdisplaybreaks

\newcommand\omycite[1]{}

\newcommand{\beqs}{\begin{equation*}}
\newcommand{\eeqs}{\end{equation*}}
\newcommand{\beq}{\begin{equation}}
\newcommand{\eeq}{\end{equation}}

\begin{document}
\title[On the total number of ones associated with cranks of partitions modulo 11]{On the total number of ones associated with cranks of partitions modulo 11}

\author{Dandan Chen}
\address{Department of Mathematics, Shanghai University, People's Republic of China}
\address{Newtouch Center for Mathematics of Shanghai University, Shanghai, People's Republic of China}
\email{mathcdd@shu.edu.cn}
\author{Rong Chen}
\address{Department of Mathematics, Shanghai Normal University, Shanghai, China}
\email{rchen@shnu.edu.cn}
\author{Siyu Yin}
\address{Department of Mathematics, Shanghai University, People's Republic of China}
\email{siyuyin@shu.edu.cn, siyuyin0113@126.com}



\date{}



\begin{abstract}
In 2021, Andrews mentioned that George Beck introduced partition statistics $M_w(r,m,n)$, which denote the total number of ones in the partition of $n$ with crank congruent to $r$ modulo $m$. Recently, a number of congruences and identities involving $M_w(r,m,n)$ for some small $m$ have been developed.  We  establish the 11-dissection of the generating functions for $M_{\omega}(r,11,n)-M_{\omega}(11-r,11,n)$, where $r=1,2,3,4,5$. In particular, we  discover a beautiful identity involving $M_{\omega}(r,11,11n+6)$.
\end{abstract}

\maketitle

\section{Introduction}
A partition of a positive integer $n$ is defined as a sequence of positive integers in non-increasing order that sum to $n$. Denoted by $p(n)$, it represents the count of such partitions for $n$. The following distinguished congruences were discovered by Ramanujan in \cite{Ramanujan-pn-1919}:
\begin{align}
p(5n+4)\equiv 0 &\pmod 5,\label{1.1}\\
p(7n+5)\equiv 0 &\pmod 7,\label{1.2}\\
p(11n+6)\equiv 0 &\pmod {11}\label{1.3},
\end{align}
for every non-negative integer $n$.

In an effort to provide purely combinatorial interpretations of Ramanujan's famous congruences \eqref{1.1}-\eqref{1.3}, two important partition statistics of ordinary partition-rank and crank-were introduced by Dyson \cite{Dyson-1944}, as well as by Andrews and Garvan \cite{Andrews-1988}. In 1944, Dyson \cite{Dyson-1944} defined the rank of a partition to be the largest part of the partition minus the number of parts. Dyson also conjectured that for $n\geq 0$,
$$
N(0,a,an+b)=N(1,a,an+b)=\cdots=N(a-1,a,an+b)=\frac{p(an+b)}{a},
$$
where $N(r,m,n)$ counts the number of partitions of $n$ with rank congruent to $r$ modulo $m$ and $(a,b)\in\{(5,4),(7,5)\}$. In 1954, Dyson's conjectures were proved by Atkin and Swinnerton-Dyer \cite{Swinnerton-Dyer-1954}. Therefore, the rank of partitions provided purely combinatorial descriptions of Ramanujan's congruences \eqref{1.1} and \eqref{1.2}. Unfortunately, Dyson's rank failed to account for Ramanujan's third congruence \eqref{1.3} combinatorially. Dyson conjectured the existence of an unknown partition statistic, which he whimsically called `the crank', to explain Ramanujan's third congruence modulo 11. The crank was found by Andrews and Garvan \cite{Andrews-1988} who defined the \text{crank} as the largest part, if the partition has no ones, and otherwise as the difference between the number of parts larger than the number of ones and the number of ones.
In 1987, Garvan \cite{Garvan-1988} proved that
$$
M(0,s,sn+t)=M(1,s,sn+t)=\cdots=M(s-1,s,sn+t)=\frac{p(sn+t)}{s},
$$
where $M(r,m,n)$ counts the number of partitions of $n$ with crank congruent to $r$ modulo $m$, and $(s,t)\in\{(5,4),(7,5),(11,6)\}$. Garvan's results implied that the crank of partitions provided purely combinatorial descriptions of Ramanujan's three congruences \eqref{1.1}-\eqref{1.3}.

In 2021, two partitions statistics, $NT(r,m,n)$ and $M_{\omega}(r,m,n)$, were proposed by Andrews in \cite{Andrews-2021}. These statistics count the total number of parts in the partitions of $n$ with rank congruent to $r$ modulo $m$, and the total number of ones in the partitions of $n$ with crank congruent to $r$ modulo $m$, respectively. They were introduced by George Beck, namely,
\begin{align*}
NT(r,m,n)=\sum_{\substack {\pi\vdash n,\\ \text{rank}(\pi)\equiv r \pmod m}} \sharp(\pi)
\end{align*}
and
\begin{align*}
M_{\omega}(r,m,n)=\sum_{\substack{\pi\vdash n,\\ \text{crank}(\pi)\equiv r \pmod m}} \omega(\pi),
\end{align*}
where $\sharp(\pi)$ denotes the number of parts of $\pi$, and $\omega(\pi)$ the number of ones.
The following Andrews-Beck type congruence, conjectured by Beck, was proved by Andrews \cite{Andrews-2021}:
$$
\sum_{m=1}^4mNT(m,5,5n+1)\equiv\sum_{m=1}^4mNT(m,5,5n+4)\equiv 0\pmod 5.
$$

Inspired by Andrews' work, Chern \cite{S.Chern-accepted,S.Chern-2022-RJ,S.Chern-2022-NT} proved some identities involving the weighted rank and crank moments, and established a number of identities and congruences on $NT(r,m,n)$ and $M_{\omega}(r,m,n)$, such as
\begin{align*}
\sum_{m=1}^4mM_{\omega}(m,5,5n+4)\equiv0\pmod 5.
\end{align*}
In 2021, Chan, Mao, and Osburn \cite{chan-2021} proved three variations of Andrews-Beck type congruences and posed a number of conjectures on $NT(r,m,n)$ and $M_{\omega}(r,m,n)$. Those conjectures regarding the Andrews-Beck type congruences for $NT(r,m,n)$ and $M_{\omega}(r,m,n)$ were subsequently proved by Chern \cite{S.Chern-2022-NT}. Additionally, Mao \cite{Mao-JMAA-2022} conjectured five identities on $NT(r,m,n)$ and $M_{\omega}(r,m,n)$. Three of them were subsequently proved by Jin, Liu, and Xia \cite{Xia-Mw-5-2022}, as well as by Mao and Xia \cite{Mao-Xia-Mw-3-2023}. Specifically, Mao demonstrated that the generating functions for total number of parts in partitions, associated with ranks modulo $7$, turn out to be theta-functions. Some identities for $M_{\omega}(b,7,7n+s)-M_{\omega}(7-b,7,7n+s)$  were proved by Xia, Yan, and Yao in \cite{Xia-Mw-7-2023}, such as
\begin{align}
\label{Xia-eta-1}\sum_{n=0}^{\infty}[M_{\omega}(2,7,7n+5)&-M_{\omega}(5,7,7n+5)\\&+4M_{\omega}(3,7,7n+5)-4M_{\omega}(4,7,7n+5)]q^n
=-7\frac{J_7^5J_{2,7}}{J_{1,7}^2J_{3,7}}\nonumber,\\
\label{Xia-eta-2}\sum_{n=0}^{\infty}[M_{\omega}(1,7,7n+5)&-M_{\omega}(6,7,7n+5)\\&+2M_{\omega}(3,7,7n+5)-2M_{\omega}(4,7,7n+5)]q^n
=-7\frac{J_7^5J_{3,7}}{J_{1,7}J_{2,7}^2}\nonumber.
\end{align}
Additionally, Mao and Xia had provided the generating function for $M_{\omega}(1,3,3n)-M_{\omega}(2,3,3n)$ and $M_{\omega}(1,4,2n)-M_{\omega}(3,4,2n)$ in the form of eta-products, as presented in \cite{Mao-Xia-Mw-3-2023} and \cite{Mao-Xia-Mw-4-2024}, respectively.

Considering the equation
\begin{align*}
\sum_{m=1}^{{(p-1)/2}}m\left[M_{\omega}(m,p,pn-\delta_p)-M_{\omega}(p-m,p,pn-\delta_p)\right]=0;
\end{align*}
where $\delta_p=\frac{p^2-1}{24}$. The case of $p=5$ was proved by Jin, Liu, and Xia in \cite{Xia-Mw-5-2022}. For the case of $p=7$, utilizing the method mentioned in \cite{Xia-Mw-7-2023}, we can obtain that
\begin{align}
\sum_{n=0}^{\infty}[3M_{\omega}(1,7,7n+5)&-3M_{\omega}(6,7,7n+5)\\ \nonumber
&-M_{\omega}(3,7,7n+5)+M_{\omega}(4,7,7n+5)]q^n=14q\frac{J_7^5J_{1,7}}{J_{2,7}J_{3,7}^2}.
\end{align}
With \eqref{Xia-eta-1} and \eqref{Xia-eta-2}, we prove the case of 7. In this paper, we prove its case of $p=11$. Motivated by the works of Chern, Mao, and Xia, we establish the 11-dissections of
$$
\sum_{n=0}^{\infty}(M_{\omega}(b,11,n)-M_{\omega}(11-b,11,n))q^n
$$
for $b=1,2,3,4,5$. We also derive an interesting equality as follows.
\begin{theorem}\label{main-thm}
For any nonnegative integer $n$, we have
$$
\sum_{m=1}^{5}m\left[M_{\omega}(m,11,11n+6)-M_{\omega}(11-m,11,11n+6)\right]=0.
$$
\end{theorem}
Based on the Appendix, we can directly obtain the following identities, which are analogous to \eqref{Xia-eta-1}.
\begin{cor}
Define $$f_m(b)=\sum_{n=0}^{\infty}(M_{\omega}(b,11,11n+m)-M_{\omega}(11-b,11,11n+m))q^n.$$ We have
\begin{align*}
&f_0(1)-f_0(3)-2f_0(4)+2f_0(5)=-11q\frac{J_{5,11}J_{11}^5}{J_{1,11}J_{3,11}J_{4,11}},\\
&f_2(1)+f_2(3)-f_2(4)=-11q\frac{J_{11}^5}{J_{2,11}^2},\\
&f_5(1)-2f_5(3)+f_5(5)=-11q\frac{J_{3,11}J_{11}^5}{J_{1,11}J_{4,11}^2},\\
&2f_7(3)+f_7(5)=-11\frac{J_{3,11}J_{4,11}J_{11}^5}{J_{1,11}J_{2,11}^2J_{5,11}},\\
&f_{10}(1)+f_{10}(2)+2f_{10}(4)=-11\frac{J_{5,11}J_{11}^5}{J_{2,11}^2J_{3,11}}.
\end{align*}
\end{cor}

The paper is organized as follows. In Section \ref{sec-pre}, we introduce some basic notations, modular functions, and several important lemmas. In Section \ref{proof} we demonstrate the proof of Theorem \ref{main-thm} using the method mentioned in \cite{Chen-maple-2020} and \cite{Mao-main-eta-2022}. Additionally, we obtain the 11-dissection of the generating function
$$
\sum_{n=0}^{\infty}(M_{\omega}(b,11,n)-M_{\omega}(11-b,11,n))q^n
$$
for $b=1,2,3,4,5$, which is demonstrated in the Appendix.

\section{Preliminaries}\label{sec-pre}

\subsection{Notations}

In this section, we first introduce the notations used in this paper, some of which are also mentioned in \cite{Mao-main-eta-2022}. For $1 \leq m \leq k$ and $|q|\textless 1$, we define
\begin{align}
&(x)_{\infty}:=(x;q)_{\infty}:=\prod_{k=0}^{\infty}(1-xq^k),\nonumber\\
&(x_1,\cdots ,x_m)_{\infty}:=(x_1,\cdots ,x_m;q)_{\infty}:=(x_1;q)_{\infty}\cdots (x_m;q)_{\infty},\nonumber\\
&[x_1,\cdots ,x_m]_{\infty}:=[x_1,\cdots ,x_m;q]_{\infty}:=(x_1,q/x_1,\cdots ,x_m,q/x_m;q)_{\infty},\nonumber\\
&J_{a,b}:=(q^a,q^{b-a},q^b;q^b)_{\infty},\nonumber\\
&J_a:=J_{a,3a}=(q^a;q^a)_{\infty},\nonumber\\
&X(a;q):=\sum_{k=0}^{\infty}(\frac{aq^k}{1-aq^k}-\frac{q^{k+1}/a}{1-q^{k +1}/a}),\label{X-def}\\
&H(a;q):=\sum_n \frac{aq^n}{(1-aq^n)^2},\label{H-def}\\
&S_0(a,k,l,j):={\sum_n}' \frac{(-1)^{kn}q^{\frac{kn(lkn+1)}{2}+akn}}{(1-q^{k^2n})^j},\label{S0-def}\\
&S_m(a,k,l,j):=\sum_n\frac{(-1)^{kn+m}q^{\frac{(kn+m)(l(kn+m)+1)}{2}+a(kn+m)}}{(1-q^{k(kn+m)})^j}.\label{Sm-def}
\end{align}
Besides, in order to simplify the notations used below, we introduce some abbreviations:
\begin{align*}
&X(a):=X(q^{11a};q^{121}), \text{where $a$ is an integer not divisible by $11$};\\
&{\sum_n}:\text{summation of the formula over $n$ from $-\infty$ to $\infty$};\\
&{\sum_n}':\text{summation of the formula over $n$ from $-\infty$ to $\infty$ but excluding $n=0$};\\
&idem(a_1;a_2,\cdots,a_{r+1}):\text{the summation of the $r$ expressions obtained from the preceding} \\
&\text{expression by interchanging $a_1$ with each $a_k$, $k=2,3,\cdots,r+1$};\\
&\mathcal{F}(b):=\sum_{n=0}^{\infty}(M_{\omega}(b,11,n)-M_{\omega}(11-b,11,n))q^{n};\\
&F_m(b):=\sum_{n=0}^{\infty}(M_{\omega}(b,11,11n+m)-M_{\omega}(11-b,11,11n+m))q^{11n}.
\end{align*}

\subsection{Modular functions}
According to \cite{Berndt-III} and \cite{Ra77}, let $\mathcal{H}=\{\tau :\Im(\tau)\textgreater 0\}$ denote the complex upper half-plane. For each $M=\left(\begin{array}{cc} a & b \\c & d \end{array} \right)\in M_2^{+}(\mathbb{Z})$, where $M_2^{+}(\mathbb{Z})$ is the set of integer $2\times 2$ matrices with positive determinant, the bilinear transformation $M(\tau)$ is defined by
\begin{align*}
M\tau=M(\tau)=\frac{a\tau+b}{c\tau+d}.
\end{align*}
The slash operator is defined by
\begin{align*}
(f|M)(\tau)=f(M\tau),
\end{align*}
and satisfies
\begin{align*}
f|MS=f|M|S,
\end{align*}
for matrices $M$ and $S$. The modular group $\Gamma(1)$ is defined by
\begin{align*}
\Gamma(1)=\left\{\left(\begin{array}{cc} a & b \\c & d \end{array} \right)\in M_2^{+}(\mathbb{Z}):ad-bc=1 \right\}.
\end{align*}
We consider the following subgroup $\Gamma$ of the modular group with finite index
\begin{align*}
\Gamma_1(N)=\left\{\left(\begin{array}{cc} a & b \\c & d \end{array} \right)\in \Gamma(1): \left(\begin{array}{cc} a & b \\c & d \end{array} \right)\equiv \left(\begin{array}{cc} 1 & * \\0 & 1 \end{array} \right) \pmod N      \right\}.
\end{align*}
The generalized Dedekind eta function is defined to be
\begin{align}\label{generalized-eta}
\eta_{\delta,g}(\tau)=q^{\frac{\delta}{2}P_2(g/\delta)}\prod_{m\equiv\pm g\pmod \delta}(1-q^m),
\end{align}
where $P_2(t)=\{t\}^2-\{t\}+\frac{1}{6}$ is the second periodic Bernoulli polynomial, $\{t\}=t-[t]$ is the fractional part of $t$, $g$, $\delta$, $m\in\mathbb{Z}^{+}$ and $0\textless g\textless \delta$. The function $\eta_{\delta,g}(\tau)$ is a modular function on $SL_2(\mathbb{Z})$ with a multiplier system. Let $N$ be a fixed positive integer. A generalized Dedekind eta-product of level $N$ has the form
\begin{align}\label{f-tau}
f(\tau)=\prod_{\delta|N,~0\textless g\textless\delta}\eta_{\delta,g}^{r_{\delta,g}}(\tau),
\end{align}
where
\begin{align}\label{r}
r_{\delta,g}\in \left\{\begin{array}{ll}
			\frac{1}{2}\mathbb{Z}, ~&g=\frac{\delta}{2},\\
			\mathbb{Z},~&otherwise.\end{array}\right.
\end{align}
Robins \cite{R094} has found sufficient conditions under which a generalized eta-product is a modular function on $\Gamma_1(N)$.
\begin{theorem}\cite[Theorem 3]{R094}
The function $f(\tau)$, defined in \eqref{f-tau}, is a modular function on $\Gamma_1(N)$ if
\begin{align*}
(i)&\sum_{\delta|N,~g}\delta P_2(\frac{g}{\delta})r_{\delta,g} \equiv 0 \pmod 2,~and\\
(ii)&\sum_{\delta|N,~g}\frac{N}{\delta}P_2(0)r_{\delta,g} \equiv 0 \pmod 2.
\end{align*}.
\end{theorem}
Cho, Koo and Park \cite{CKP09} have found a set of inequivalent cusps for $\Gamma_1(N)\cap\Gamma_0(mZ)$. The group $\Gamma_1(N)$ corresponds to the case $m=1$.

\begin{theorem}\cite[Corollary 4]{CKP09}\label{Thm-2.10}
Let $a$, $c$, $a'$, $c'\in\mathbb{Z}$ with ($a$, $c$)=($a'$, $c'$)=1.\\
(i)The cusps $\frac{a}{c}$ and $\frac{a'}{c'}$ are equivalent mod $\Gamma_1(N)$ if and only if
$$
\binom{a'}{c'}\equiv \pm\binom{a+nc}{c} \pmod N
$$
for some integers $n$.\\
(ii)The following is a complete set of inequivalent cusps mod $\Gamma_{1}$.
\begin{align*}
\mathcal{S}=&\Bigg\{\frac{y_{c,j}}{x_{c,j}}:0\textless c|N,0\textless s_{c,i},a_{c,j}\leq N,(s_{c,i},N)=(a_{c,j},N)=1,\\
&s_{c,i}=s_{c,i'}\Leftrightarrow s_{c,1}\equiv\pm s_{c',i'}\pmod {\frac{N}{c}},\\
&a_{c,j}=a_{c,j'}\Leftrightarrow \left\{\begin{array}{ll}
a_{c,j}\equiv\pm a_{c,j'}\pmod c, ~\text{if}~c=\frac{N}{2}~\text{or}~ N,\\
a_{c,j}\equiv\pm a_{c,j'}\pmod c, ~\text{otherwise}.\end{array}
\right.\\
&x_{c,i},y_{c,j}\in\mathbb{Z}~\text{chosen so that}~x_{c,i}\equiv cs_{c,i} , y_{c,j}\equiv a_{c,j}\pmod N, (x_{c,i},y_{c,j})=1.\Bigg\}
\end{align*}
(iii)The fan width of the cusp $\frac{a}{c}$ is given by
\begin{align*}
\kappa(\frac{a}{c},\Gamma_1(N))=\left\{\begin{array}{ll}
1,~&\text{if}~N=4~\text{and}~(c,4)=2,\\
\frac{N}{(c,N)},~&\text{otherwise}.\end{array}
\right.
\end{align*}
\end{theorem}
In this theorem, it is understood, as usual that the fraction $\frac{\pm 1}{0}$ corresponds to $\infty$. Robins \cite{R094} has calculated the invariant order of $\eta_{\delta,g}(\tau)$ at any cusp. This gives a method for calculating the invariant order at any cusp of a generalized eta-product.
\begin{theorem}\cite{R094}\label{Thm-2.11}
The order at the cusp $\zeta=\frac{a}{c}$ (assuming $(a,c)=1$) of the generalized eta-function $\eta_{\delta,g}(\tau)$ (defined in \eqref{generalized-eta} and assuming  $0\textless g\textless \delta$ ) is
\begin{align}
ord(\eta_{\delta,g}(\tau);\zeta)=\frac{\varepsilon^2}{2\delta}P_2(\frac{ag}{\varepsilon}),
\end{align}
where $\varepsilon=(\delta,c).$
\end{theorem}

\begin{theorem}\cite[Corollary 2.5]{FrGa19}\label{Thm-2.12}
Let $f_1(\tau)$, $f_2(\tau)$, $\cdots$, $f_n(\tau)$ be generalized eta-products that are modular functions on $\Gamma_1(N)$. Let $\mathcal{S}_N$ be a set of inequivalent cusps for $\Gamma_1(N)$. Define the constant
\begin{align}\label{B}
B=\sum_{s\in\mathcal{S}_N,~s\neq\infty}\min(\{Ord(f_j,s,\Gamma_1(N)):1\leq j\leq n\}\cup\{0\}),
\end{align}
and consider
\begin{align}\label{f-tau-1}
g(\tau):=\alpha_1f_1(\tau)+\alpha_2f_2(\tau)+\cdots+\alpha_nf_n(\tau)+1,
\end{align}
where each $\alpha_j\in\mathbb{C}$. Then
$$
g(\tau)\equiv 0
$$
if and only if
$$
Ord(g(\tau),\infty,\Gamma_1(N))\textgreater -B.
$$
\end{theorem}
A more comprehensive description can be found in \cite{FrGa19}. We have utilized a MAPLE package known as thetaids, which facilitates the implementation of the aforementioned algorithm. See
\begin{align*}
\href{http://qseries/org/fgarvan/qmaple/thetaids/}{http://qseries/org/fgarvan/qmaple/thetaids/}
\end{align*}

\subsection{Some lemmas}
In the subsequent lemmas, we establish key identities that are instrumental for the proof of Theorem \ref{main-thm}. Lemma \ref{s1-eta} is employed to derive identities involving $S_m(a,k,l,j)$ with $j=1$, while Lemma \ref{s2-eta} is utilized to construct identities involving $S_m(a,k,l,j)$ with $j=2$.

\begin{lemma}\label{s1-eta}
We have
\begin{align}
&\frac{J_1^2}{[b_1]_{\infty}}=\sum_k \frac{(-1)^kq^{\frac{k(k+1)}{2}}}{1-b_1q^k},\label{L-2-1}\\
&\frac{[a_1]_{\infty} J_1^2}{[b_1,b_2]_{\infty}}=\frac{[a_1/b_1]_{\infty}}{[b_2/b_1]_{\infty}}\sum_k \frac{(-1)^kq^{\frac{k(k+1)}{2}}}{1-b_1q^k}(\frac{a_1}{b_2})^k+idem(b_1;b_2)\label{L-2-2},\\
&\frac{[a_1]_{\infty}}{[b_1]_{\infty}}\{X(a_1;q)-X(b_1;q)\}\label{L-2-3}\\
=&\frac{[a_1/b_1]_{\infty}}{[1/b_1]_{\infty}}\sum_k \frac{(-1)^kq^{\frac{k(k+1)}{2}}}{1-b_1q^k}{a_1}^k+\frac{[a_1]_{\infty}}{[b_1]_{\infty}}{\sum_k}' \frac{(-1)^kq^{\frac{k(k+1)}{2}}}{1-q^k}(\frac{a_1}{b_1})^k.\nonumber
\end{align}
\end{lemma}
\begin{proof}
The first two identities correspond to the special cases with $r=0$, $s=1$ and $r=1$, $s=2$, respectively, as presented in \cite[Theorem 2.1]{chan-2005}. By applying the operator [$\frac{\partial}{\partial b_2}(1-b_2)|_{b_2=1}$] to both sides of \eqref{L-2-2} and logarithmically differentiating the infinite products, we obtain \eqref{L-2-3} after some lengthy but straightforward calculations.
\end{proof}

\begin{lemma}\label{s2-eta}
We obtain that
\begin{align}
&\frac{J_1^4}{[b_1,b_2]_{\infty}}=\frac{[1/b_1]_{\infty}}{[b_2/b_1]_{\infty}}\sum_k \frac{(-1)^{k+1}q^{\frac{k(k+1)}{2}}b_1}{(1-b_1q^k)^2}(\frac{q}{b_2})^k+idem(b_1;b_2),\label{L-3-1}\\
&\frac{J_1^2}{[b_1]_{\infty}}X(b_1;q)=\sum_k \frac{(-1)^{k}q^{\frac{k(k+1)}{2}}}{(1-b_1q^k)^2}b_1q^k,\label{L-3-2}\\
&\frac{J_1^4[a_2]_{\infty}}{[b_1,b_2,b_3]_{\infty}}=\frac{[1/b_1,a_2/b_1]_{\infty}}{[b_2/b_1,b_3/b_1]_{\infty}}\sum_k \frac{(-1)^{k+1}q^{\frac{k(k+1)}{2}}b_1}{(1-b_1q^k)^2}(\frac{a_2q}{b_2b_3})^k+idem(b_1;b_2,b_3),\label{L-3-3}\\
&\frac{[a_2/b_1]_{\infty}}{[b_2/b_1]_{\infty}}\sum_k\frac{(-1)^{k}q^{\frac{k(k+1)}{2}}b_1}{(1-b_1q^k)^2}(\frac{a_2}{b_2}q)^k\label{L-3-4}
+\frac{[a_2/b_2]_{\infty}}{[b_1/b_2]_{\infty}}\sum_k\frac{(-1)^{k}q^{\frac{k(k+1)}{2}}b_2}{(1-b_2q^k)^2}(\frac{a_2}{b_1}q)^k\\ \nonumber
=&-\frac{J_1^2[a_2]_{\infty}}{[b_1,b_2]_{\infty}}\{X(a_2;q)-X(b_1;q)-X(b_2;q)\},\\
&\frac{[a_2]_{\infty}}{[b_1]_{\infty}}{\sum_k}'\frac{(-1)^{k}q^{\frac{k(k+1)}{2}}}{(1-q^k)^2}(\frac{a_2q}{b_1})^k\label{L-3-5}
+\frac{[a_2/b_1]_{\infty}}{[1/b_1]_{\infty}}\sum_k\frac{(-1)^{k}q^{\frac{k(k+1)}{2}}b_1}{(1-b_1q^k)^2}(a_2q)^k\\ \nonumber
&=-\frac{[a_2]_{\infty}}{2[b_1]_{\infty}}\left\{-\sum_{n=1}^{\infty}\frac{2q^n}{(1-q^n)^2}+\mathcal{S}_1(a_2,b_1;q)[\mathcal{S}_1(a_2,b_1;q)+2]+\mathcal{S}_2(a_2,b_1;q)\right\}
\end{align}
where
\begin{align*}
&\mathcal{S}_1(a_2,b_1;q)=X(a_2;q)-X(b_1;q),\\
&\mathcal{S}_2(a_2,b_1;q)=-X(a_2;q)+X(b_1;q)-H(a_2;q)+H(b_1;q).
\end{align*}
\end{lemma}
\begin{proof}
Recall  \cite[Theorem 2.3]{chan-2005}: for $r\textless s$,
\begin{align*}
&\frac{(a_1q,q/a_1;q)_{\infty}J_1^2[a_2,\cdots,a_r]_{\infty}}{[b_1,b_2,\cdots,b_s]_{\infty}}\\
=&\frac{[a_1/b_1,\cdots,a_r/b_1]_{\infty}}{[b_2/b_1,\cdots,b_s/b_1]_{\infty}}\sum_k\frac{(-1)^{(s-r)k+1}q^{(s-r)k(k+1)/2}b_1a_1^{-1}}{(1-b_1q^k)(1-b_1q^k/a_1)}\left(\frac{a_1\cdots a_sb_1^{s-r-1}q}{b_2\cdots b_s}\right)^k\\
&+idem(b_1;b_2,\cdots,b_s).
\end{align*}
Taking $r=1$, $s=2$ and $a_1=1$ in the above identity, we deduce equation \eqref{L-3-1}. Similarly, equation \eqref{L-3-3} emerges as a special case with $r=2$, $s=3$, and $a_1=1$. Applying the operator [$\frac{\partial^2}{\partial b_2^2}(1-b_2)^2|_{b_2=1}$] to both sides of \eqref{L-3-1} and simplifying, we arrive at \eqref{L-3-2}.

Applying the operator [$\frac{\partial^2}{\partial b_3^2}(1-b_3)^2|_{b_3=1}$] to both sides of \eqref{L-3-3} and simplifying, we obtain \eqref{L-3-4}. Subsequently, applying the operator [$\frac{\partial^2}{\partial b_2^2}(1-b_2)^2|_{b_2=1}$] to both sides of \eqref{L-3-4} and simplifying, we derive \eqref{L-3-5}.
\end{proof}

\begin{lemma}\label{X-eta}
Recall that $X(a;q)$ is defined in \eqref{X-def}, we have
\begin{align*}
X(q;q^{11})=\frac{1}{242}&\left(\frac{81J_{2,11}^3J_{11}^3}{J_{1,11}^3J_{3,11}}-\frac{9J_{4,11}^3J_{11}^3}{J_{2,11}^3J_{5,11}}+\frac{27J_{5,11}^3J_{11}^3}{J_{3,11}^3J_{2,11}}
-\frac{qJ_{3,11}^3J_{11}^3}{J_{4,11}^3J_{1,11}}-\frac{3q^{4}J_{1,11}^3J_{11}^3}{J_{5,11}^3J_{4,11}}-99\right),\nonumber\\
X(q^{2};q^{11})=\frac{1}{242}&\left(-\frac{3J_{2,11}^3J_{11}^3}{J_{1,11}^3J_{3,11}}+\frac{81J_{4,11}^3J_{11}^3}{J_{2,11}^3J_{5,11}}-\frac{J_{5,11}^3J_{11}^3}{J_{3,11}^3J_{2,11}}
+\frac{9qJ_{3,11}^3J_{11}^3}{J_{4,11}^3J_{1,11}}+\frac{27q^{4}J_{1,11}^3J_{11}^3}{J_{5,11}^3J_{4,11}}-77\right),\nonumber\\
X(q^{3};q^{11})=\frac{1}{242}&\left(\frac{J_{2,11}^3J_{11}^3}{J_{1,11}^3J_{3,11}}-\frac{27J_{4,11}^3J_{11}^3}{J_{2,11}^3J_{5,11}}+\frac{81J_{5,11}^3J_{11}^3}{J_{3,11}^3J_{2,11}}
-\frac{3qJ_{3,11}^3J_{11}^3}{J_{4,11}^3J_{1,11}}-\frac{9q^{4}J_{1,11}^3J_{11}^3}{J_{5,11}^3J_{4,11}}-55\right),\nonumber\\
X(q^{4};q^{11})=\frac{1}{242}&\left(\frac{27J_{2,11}^3J_{11}^3}{J_{1,11}^3J_{3,11}}-\frac{3J_{4,11}^3J_{11}^3}{J_{2,11}^3J_{5,11}}+\frac{9J_{5,11}^3J_{11}^3}{J_{3,11}^3J_{2,11}}
-\frac{81qJ_{3,11}^3J_{11}^3}{J_{4,11}^3J_{1,11}}-\frac{q^{4}J_{1,11}^3J_{11}^3}{J_{5,11}^3J_{4,11}}-33\right),\nonumber\\
X(q^{5};q^{11})=\frac{1}{242}&\left(\frac{9J_{2,11}^3J_{11}^3}{J_{1,11}^3J_{3,11}}-\frac{J_{4,11}^3J_{11}^3}{J_{2,11}^3J_{5,11}}+\frac{3J_{5,11}^3J_{11}^3}{J_{3,11}^3J_{2,11}}
-\frac{27qJ_{3,11}^3J_{11}^3}{J_{4,11}^3J_{1,11}}-\frac{81q^{4}J_{1,11}^3J_{11}^3}{J_{5,11}^3J_{4,11}}-11\right)\nonumber.
\end{align*}
\end{lemma}
\begin{proof}
From \cite[Eq.(2.16)]{Mao-main-eta-2022}, we have the following identity:
\begin{align}\label{X-trans}
\frac{J_1^2[x^2]_{\infty}^3 }{[x]_{\infty}^3[x^3]_{\infty}}
=1+3\sum_{n=0}^{\infty}\left(\frac{xq^n}{1-xq^n}-\frac{q^{n+1}/x}{1-q^{n+1}/x}\right)-\sum_{n=0}^{\infty}\left(\frac{x^3q^n}{1-x^3q^n}-\frac{q^{n+1}/x^3}{1-q^{n+1}/x^3}\right).
\end{align}
By replacing $q$  with $q^{11}$ and $x$ with $q$ in \eqref{X-trans}, and rearranging it into the form of $X(a;q)$, we obtain
\begin{align*}
3X(q;q^{11})-X(q^{3};q^{11})+1=\frac{J_{2,11}^3J_{11}^3}{J_{1,11}^3J_{3,11}}.
\end{align*}
Similarly, by substituting suitable values for $x$ and $q$, we derive the following four equations:
\begin{align*}
&3X(q^{2};q^{11})+X(q^{5};q^{11})+1=\frac{J_{4,11}^3J_{11}^3}{J_{2,11}^3J_{5,11}},\\
&3X(q^{3};q^{11})+X(q^{2};q^{11})+1=\frac{J_{5,11}^3J_{11}^3}{J_{3,11}^3J_{2,11}},\\
&3X(q^{4};q^{11})-X(q;q^{11})=-\frac{qJ_{3,11}^3J_{11}^3}{J_{4,11}^3J_{1,11}},\\
&3X(q^{5};q^{11})-X(q^{4};q^{11})=-\frac{q^{4}J_{1,11}^3J_{11}^3}{J_{5,11}^3J_{4,11}}.
\end{align*}
It is straightforward to deduce Lemma \ref{X-eta} from these equations.
\end{proof}

\begin{lemma}\label{S-eta}
Recall that $H(a;q)$ is defined in \eqref{H-def}, we get
\begin{align}
H(q^{b};q^{11})-H(q^{c};q^{11})=-\frac{q^{c}J_{b+c,11}J_{b-c,11}J_{11}^6}{J_{b,11}^2J_{c,11}^2},
\end{align}
where $b$, $c$, and $b-c$ are integers not divisible by 11.
\end{lemma}

\begin{proof}
Let $(r,s,q,a_1,a_2,b_1,b_2)\rightarrow(2,2,q^{11},1,q^{b+c},q^{b},q^{c})$ in \cite[Theorem 2.3]{chan-2005}, where $b$, $c$, and $b-c$ are integers not divisible by $11$. Then, we have
\begin{align}
-\frac{J_{11}^6 J_{b+c,11}}{J_{b,11}J_{c,11}}
=\frac{J_{-b,11}J_{c,11}}{J_{c-b,11}}H(q^{b};q^{11})+\frac{J_{-c,11}J_{b,11}}{J_{b-c,11}}H(q^{c};q^{11}).
\end{align}
With the help of the properties of $J_{a,b}$, we complete the proof of Lemma \ref{S-eta}.
\end{proof}

\section{The proof of Theorem 1}\label{proof}
We require the following generating functions.
\begin{lemma}\cite[p.9]{Mao-Generating-2024}\label{generating function}
For $1 \leq b \leq k-1$, we have
\begin{align}\label{Generating-1}
&\sum_{n=0}^{\infty}(M_{\omega}(b,k,n)-M_{\omega}(k-b,k,n))q^n\\
=&\frac{k}{(q)_{\infty}}{\sum_n} '\frac{(-1)^nq^{\frac{n(n+1)}{2}+(b-1)n}(1-q^n)}{(1-q^{kn})^2}
-\frac{k-b}{(q)_{\infty}}{\sum_n}'\frac{(-1)^nq^{\frac{n(n+1)}{2}+(b-1)n}(1-q^n)}{1-q^{kn}}.\nonumber
\end{align}
\end{lemma}

Here and later, we only consider the case where $k=11$ in the above equation, we can rewrite \eqref{Generating-1} as follows:
\begin{align}\label{generating-2}
\mathcal{F}(b)=&\sum_{n=0}^{\infty}(M_{\omega}(b,11,n)-M_{\omega}(11-b,11,n))q^n\\
=&\frac{11}{J_1}{\sum_n} '\frac{(-1)^nq^{\frac{n(n+1)}{2}+(b-1)n}(1-q^n)}{(1-q^{11n})^2}
-\frac{11-b}{J_1}{\sum_n}'\frac{(-1)^nq^{\frac{n(n+1)}{2}+(b-1)n}(1-q^n)}{1-q^{11n}}\nonumber\\
=&\frac{11}{J_1}\sum_{m=0}^{10}(S_m(b-1,11,1,2)-S_m(b,11,1,2))\nonumber\\
&-\frac{11-b}{J_1}\sum_{m=0}^{10}(S_m(b-1,11,1,1)-S_m(b,11,1,1)).\nonumber
\end{align}

According to Lemmas \ref{s1-eta} and \ref{s2-eta}, we derive the following two lemmas, which facilitate the transformation of \eqref{generating-2} into eta-quotients.
\begin{lemma}\label{1-eta}
Recall $S_0(a,k,l,j)$ and $S_m(a,k,l,j)$ as defined in \eqref{S0-def} and \eqref{Sm-def}, respectively, we have
\begin{align}
&{\sum_n}'\frac{(-1)^nq^{\frac{n(n+1)}{2}}}{1-q^{11n}}\label{cor-1-1}\\
=&\sum_{m=0}^{10}S_m(0,11,1,1)\nonumber \\
=&-q\frac{J_{44,121}}{J_{55,121}}[2X(4)+X(1)+X(2)+1]+q^3\frac{J_{33,121}}{J_{55,121}}[3X(3)+X(2)+1]-q^{15}\frac{J_{121}^3}{J_{55,121}}\nonumber \\
&-q^6\frac{J_{22,121}}{J_{55,121}}[2X(2)+X(3)+X(4)+1]+q^{10}\frac{J_{11,121}}{J_{55,121}}[2X(1)+X(4)+X(5)+1]\nonumber \\
&+X(1)+X(5),\nonumber \\
&{\sum_n}'\frac{(-1)^nq^{\frac{n(n+1)}{2}+n}}{1-q^{11n}}\label{cor-1-2}\\
=&\sum_{m=0}^{10}S_m(1,11,1,1)\nonumber \\
=&-q^2\frac{J_{33,121}}{J_{44,121}}[2X(3)+X(1)+X(4)+1]+q^{14}\frac{J_{121}^3}{J_{44,121}}+q^5\frac{J_{22,121}}{J_{44,121}}[3X(2)+X(5)+1]\nonumber \\
&-q^9\frac{J_{11,121}}{J_{44,121}}[2X(1)+X(3)-X(5)+1]-q^{-1}\frac{J_{55,121}}{J_{44,121}}[2X(5)+X(2)-X(1)]+X(3)+X(4),\nonumber \\
&{\sum_n}'\frac{(-1)^nq^{\frac{n(n+1)}{2}+2n}}{1-q^{11n}}\nonumber\\
=&\sum_{m=0}^{10}S_m(2,11,1,1)\nonumber \\
=&-q^{12}\frac{J_{121}^3}{J_{33,121}}-q^3\frac{J_{22,121}}{J_{33,121}}[2X(2)+X(1)-X(5)+1]+q^7\frac{J_{11,121}}{J_{33,121}}[2X(1)+X(2)-X(4)+1]\nonumber \\
&+q^{-3}\frac{J_{55,121}}{J_{33,121}}[2X(5)+X(3)-X(2)]-q^{-2}\frac{J_{44,121}}{J_{33,121}}[3X(4)-X(1)]+X(3)+X(5),\nonumber \\
&{\sum_n}'\frac{(-1)^nq^{\frac{n(n+1)}{2}+3n}}{1-q^{11n}}\nonumber\\
=&\sum_{m=0}^{10}S_m(3,11,1,1)\nonumber \\
=&-q^4\frac{J_{11,121}}{J_{22,121}}[3X(1)-X(3)+1]-q^{-6}\frac{J_{55,121}}{J_{22,121}}[2X(5)+X(4)-X(3)]+q^{-5}\frac{J_{44,121}}{J_{22,121}}[2X(4)\nonumber \\
&+X(5)-X(2)]-q^{-3}\frac{J_{33,121}}{J_{22,121}}[2X(3)-X(1)-X(5)]+q^9\frac{J_{121}^3}{J_{22,121}}+X(2)-X(4),\nonumber \\
&{\sum_n}'\frac{(-1)^nq^{\frac{n(n+1)}{2}+4n}}{1-q^{11n}}\nonumber\\
=&\sum_{m=0}^{10}S_m(4,11,1,1) \nonumber\\
=&q^{-10}\frac{J_{55,121}}{J_{11,121}}[3X(5)-X(4)]-q^{-9}\frac{J_{44,121}}{J_{11,121}}[2X(4)-X(3)-X(5)]+q^{-7}\frac{J_{33,121}}{J_{11,121}}[2X(3)-X(2)\nonumber \\
&-X(4)]-q^{5}\frac{J_{121}^3}{J_{11,121}}-q^{-4}\frac{J_{22,121}}{J_{11,121}}[2X(2)-X(1)-X(3)]+X(1)-X(2),\nonumber \\
&{\sum_n}'\frac{(-1)^nq^{\frac{n(n+1)}{2}+5n}}{1-q^{11n}}=\sum_{m=0}^{10}S_m(5,11,1,1)=0.\nonumber
\end{align}
\end{lemma}

\begin{proof}
Since the six equations mentioned above are analogous, we shall prove \eqref{cor-1-1} as a representative case. The proofs for the remaining equations follow a similar methodology, with adjustments specific to each case. By \eqref{L-2-1}, we get
\begin{align*}
S_5(0,11,1,1)=-q^{15}\frac{J_{121}^3}{J_{55,121}}.
\end{align*}
Using \eqref{L-2-3}, we obtain that
\begin{align*}
&S_1(0,11,1,1)=q\frac{J_{44,121}}{J_{55,121}}[-X(4)-X(1)-1-S_0(0,11,1,1)],\\
&S_2(0,11,1,1)=-q^3\frac{J_{33,121}}{J_{55,121}}[-X(3)-X(2)-1-S_0(0,11,1,1)],\\
&S_3(0,11,1,1)=q^6\frac{J_{22,121}}{J_{55,121}}[-X(2)-X(3)-1-S_0(0,11,1,1)],\\
&S_4(0,11,1,1)=-q^{10}\frac{J_{11,121}}{J_{55,121}}[-X(4)-X(1)-1-S_0(0,11,1,1)],\\
&S_6(0,11,1,1)=q^{10}\frac{J_{11,121}}{J_{55,121}}[X(1)+X(5)-S_0(0,11,1,1)],\\
&S_7(0,11,1,1)=-q^6\frac{J_{22,121}}{J_{55,121}}[X(2)+X(4)-S_0(0,11,1,1)],\\
&S_8(0,11,1,1)=q^3\frac{J_{33,121}}{J_{55,121}}[2X(3)-S_0(0,11,1,1)],\\
&S_9(0,11,1,1)=-q\frac{J_{44,121}}{J_{55,121}}[X(2)+X(4)-S_0(0,11,1,1)],\\
&S_{10}(0,11,1,1)=X(1)+X(5)-S_0(0,11,1,1).
\end{align*}
By summing the aforementioned equations and performing the necessary simplifications, we deduce \eqref{cor-1-1}.
\end{proof}

\begin{lemma}\label{2-eta}
Recall the definitions of $S_0(a,k,l,j)$ and $S_m(a,k,l,j)$ as presented in \eqref{S0-def} and \eqref{Sm-def}, respectively, we have
\begin{align}
&{\sum_n}'\frac{(-1)^nq^{\frac{n(n+1)}{2}}}{(1-q^{11n})^2}\label{cor-2-1}\\
=&\sum_{m=0}^{10}S_m(0,11,1,2)\nonumber \\
=&-q\frac{J_{22,121}J_{121}^3}{J_{11,121}^2}[X(5)+1]-q^{12}\frac{J_{33,121}J_{121}^3}{J_{22,121}J_{55,121}}[X(1)+X(2)+X(3)]\nonumber \\
&-q^3\frac{J_{33,121}J_{44,121}J_{121}^3}{J_{11,121}J_{22,121}J_{55,121}}[X(1)-X(2)-X(4)-1]\nonumber\\
&+q^{14}\frac{J_{22,121}^2J_{121}^3}{J_{11,121}J_{33,121}J_{55,121}}[X(1)+X(2)+X(3)]-q^{15}\frac{J_{121}^3}{J_{55,121}}[X(5)+1]\nonumber \\
&+q^{6}\frac{J_{44,121}J_{121}^3}{J_{11,121}J_{55,121}}[X(1)-2X(3)-1]-q^{17}\frac{J_{33,121}J_{121}^3}{J_{44,121}J_{55,121}}[2X(1)+X(4)]\nonumber \\
&+q^{10}\frac{J_{22,121}J_{121}^3}{J_{11,121}J_{44,121}}[-X(1)+X(2)+X(4)+1]+q^{21}\frac{J_{44,121}J_{121}^6}{J_{11,121}J_{55,121}^2}-\sum_{n=1}^{\infty}\frac{q^{121n}}{(1-q^{121n})^2}\nonumber\\
&-\frac{[X(1)+X(5)]^2}{2}-\frac{q^{11}J_{44,121}J_{121}^6}{2J_{11,121}^2J_{55,121}}+\frac{X(1)+X(5)}{2},\nonumber\\
&{\sum_n}'\frac{(-1)^nq^{\frac{n(n+1)}{2}+n}}{(1-q^{11n})^2}\label{cor-2-2}\\
=&\sum_{m=0}^{10}S_m(1,11,1,2)\nonumber \\
=&q^2\frac{J_{55,121}J_{121}^3}{J_{11,121}J_{33,121}}[X(5)+X(3)-X(1)-1]-q^{35}\frac{J_{11,121}J_{121}^6}{J_{33,121}J_{44,121}^2}+q^{14}\frac{J_{121}^3}{J_{44,121}}[X(4)+1]\nonumber \\
&-q^5\frac{J_{55,121}^2J_{121}^3}{J_{22,121}J_{33,121}J_{44,121}}[X(3)-X(2)-X(5)-1]-q^{27}\frac{J_{11,121}J_{22,121}J_{121}^3}{J_{33,121}J_{44,121}J_{55,121}}[X(3)+X(5)\nonumber\\
&-X(1)-1]-q^{9}\frac{J_{55,121}J_{121}^3}{J_{33,121}^2}[X(4)+1]+q^{20}\frac{J_{22,121}J_{121}^3}{J_{44,121}J_{55,121}}[X(3)-X(2)-X(5)-1]\nonumber\\
&+q^{21}\frac{J_{11,121}J_{121}^3}{J_{33,121}J_{44,121}}[2X(2)+X(3)]-q^{10}\frac{J_{22,121}J_{121}^3}{J_{11,121}J_{44,121}}[2X(3)+X(1)]-\sum_{n=1}^{\infty}\frac{q^{121n}}{(1-q^{121n})^2}\nonumber \\
&-\frac{[X(3)+X(4)]^2}{2}-\frac{q^{33}J_{11,121}J_{121}^6}{2J_{33,121}^2J_{44,121}}+\frac{X(3)+X(4)}{2},\nonumber\\
&{\sum_n}'\frac{(-1)^nq^{\frac{n(n+1)}{2}+2n}}{(1-q^{11n})^2}\nonumber\\
=&\sum_{m=0}^{10}S_m(2,11,1,2)\nonumber \\
=&-q^{12}\frac{J_{121}^3}{J_{33,121}}[X(3)+1]+q^3\frac{J_{44,121}J_{121}^3}{J_{11,121}J_{33,121}}[X(4)+X(5)-X(1)-1]-q^{25}\frac{J_{11,121}J_{121}^3}{J_{55,121}^2}[X(3)\nonumber \\
&+1]-q^{7}\frac{J_{44,121}J_{121}^3}{J_{22,121}J_{33,121}}[2X(5)-X(2)-1]-q^{18}\frac{J_{22,121}J_{121}^3}{J_{33,121}J_{55,121}}[X(5)-2X(4)-1]\nonumber\\
&+q^{19}\frac{J_{11,121}J_{121}^3}{J_{22,121}J_{55,121}}[X(1)+X(2)+X(5)]-q^{30}\frac{J_{22,121}J_{121}^6}{J_{33,121}^2J_{55,121}}-q^{9}\frac{J_{22,121}J_{44,121}J_{121}^3}{J_{11,121}J_{33,121}J_{55,121}}[X(1)\nonumber \\
&+X(2)+X(5)]-q^{31}\frac{J_{11,121}^2J_{121}^3}{J_{33,121}J_{44,121}J_{55,121}}[X(4)+X(5)-X(1)-1]-\sum_{n=1}^{\infty}\frac{q^{121n}}{(1-q^{121n})^2}\nonumber\\
&-\frac{[X(3)+X(5)]^2}{2}+\frac{q^{33}J_{22,121}J_{121}^6}{2J_{55,121}^2J_{33,121}}+\frac{X(3)+X(5)}{2},\nonumber\\
&{\sum_n}'\frac{(-1)^nq^{\frac{n(n+1)}{2}+3n}}{(1-q^{11n})^2}\nonumber\\
=&\sum_{m=0}^{10}S_m(3,11,1,2)\nonumber \\
=&q^{4}\frac{J_{33,121}^2J_{121}^3}{J_{11,121}J_{22,121}J_{44,121}}[X(3)-X(1)-X(4)-1]+q^{15}\frac{J_{11,121}J_{55,121}J_{121}^3}{J_{22,121}J_{33,121}J_{44,121}}[X(5)-X(4)\nonumber \\
&-X(3)-1]+q^{16}\frac{J_{33,121}J_{121}^3}{J_{44,121}^2}[X(2)+1]+q^{16}\frac{J_{11,121}J_{121}^3}{J_{22,121}J_{33,121}}[X(3)-X(1)-X(4)-1]\nonumber\\
&+q^{17}\frac{J_{55,121}J_{121}^6}{J_{22,121}^2J_{44,121}}+q^{17}\frac{J_{33,121}J_{121}^3}{J_{44,121}J_{55,121}}[X(5)-X(4)-X(3)-1]-q^{8}\frac{J_{55,121}J_{121}^3}{J_{22,121}J_{44,121}}[2X(1)\nonumber \\
&-X(4)]+q^{19}\frac{J_{11,121}J_{121}^3}{J_{22,121}J_{55,121}}[2X(4)+X(5)+1]+q^{9}\frac{J_{121}^3}{J_{22,121}}[X(2)+1]-\sum_{n=1}^{\infty}\frac{q^{121n}}{(1-q^{121n})^2}\nonumber\\
&-\frac{[X(2)-X(4)]^2}{2}+\frac{q^{22}J_{55,121}J_{121}^6}{2J_{22,121}J_{44,121}^2}+\frac{X(2)-X(4)}{2},\nonumber\\
&{\sum_n}'\frac{(-1)^nq^{\frac{n(n+1)}{2}+4n}}{(1-q^{11n})^2}\nonumber\\
=&\sum_{m=0}^{10}S_m(4,11,1,2)\nonumber \\
=&-q\frac{J_{33,121}J_{55,121}J_{121}^3}{J_{11,121}J_{22,121}J_{44,121}}[X(4)-X(2)-X(3)-1]+q\frac{J_{44,121}^2J_{121}^3}{J_{11,121}J_{22,121}J_{55,121}}[X(5)-X(2)\nonumber \\
&-X(4)-1]+q^{2}\frac{J_{55,121}J_{121}^3}{J_{11,121}J_{33,121}}[X(3)-2X(2)-1]+q^{2}\frac{J_{33,121}J_{121}^3}{J_{11,121}J_{22,121}}[2X(5)+X(2)+1]\nonumber\\
&+q^{4}\frac{J_{44,121}J_{121}^3}{J_{22,121}^2}[X(1)+1]+q^{4}\frac{J_{55,121}J_{121}^3}{J_{11,121}J_{44,121}}[X(5)-X(2)-X(4)-1]-q^{5}\frac{J_{121}^3}{J_{11,121}}[X(1)\nonumber\\
&+1]-q^{7}\frac{J_{44,121}J_{121}^3}{J_{22,121}J_{33,121}}[X(4)-X(2)-X(3)-1]-q^{7}\frac{J_{33,121}J_{121}^6}{J_{11,121}^2J_{22,121}}-\sum_{n=1}^{\infty}\frac{q^{121n}}{(1-q^{121n})^2}\nonumber \\
&-\frac{[X(1)+X(2)]^2}{2}+\frac{q^{11}J_{33,121}J_{121}^6}{2J_{11,121}J_{22,121}^2}+\frac{X(1)+X(2)}{2},\nonumber\\
&{\sum_n}'\frac{(-1)^nq^{\frac{n(n+1)}{2}+5n}}{(1-q^{11n})^2}\nonumber\\
=&\sum_{m=0}^{10}S_m(5,11,1,2)\nonumber \\
=&q^{13}\frac{J_{44,121}J_{121}^6}{J_{22,121}^3}-q^{6}\frac{J_{22,121}J_{121}^6}{J_{11,121}^3}-q^{40}\frac{J_{11,121}J_{121}^6}{J_{55,121}^3}+q^{30}\frac{J_{33,121}J_{121}^6}{J_{44,121}^3}-q^{21}\frac{J_{55,121}J_{121}^6}{J_{33,121}^3}-\sum_{n=1}^{\infty}\frac{q^{121n}}{(1-q^{121n})^2}.\nonumber
\end{align}
\end{lemma}

\begin{proof}
We first prove \eqref{cor-2-2}. Using \eqref{L-3-2}, we have
\begin{align*}
S_4(1,11,1,2)=\frac{q^{14}J_{121}^3}{J_{44,121}}[X(4)+1].
\end{align*}
Employing \eqref{L-3-1}, we can get
\begin{align*}
S_7(1,11,1,2)=-\frac{q^2J_{33,121}}{J_{44,121}}S_8(1,11,1,2)-\frac{q^{35}J_{11,121}J_{121}^6}{J_{33,121}J_{44,121}^2}.
\end{align*}
By means of \eqref{L-3-4}, we have
\begin{align*}
&S_1(1,11,1,2)=\frac{q^2J_{33,121}}{J_{44,121}}S_8(1,11,1,2)+\frac{q^2J_{55,121}J_{121}^3}{J_{11,121}J_{33,121}}[X(3)+X(5)-X(1)-1],\\
&S_2(1,11,1,2)=-\frac{q^5J_{22,121}}{J_{44,121}}S_8(1,11,1,2)-\frac{q^5J_{55,121}^2J_{121}^3}{J_{22,121}J_{33,121}J_{44,121}}[X(3)-X(5)-X(2)-1],\\
&S_3(1,11,1,2)=\frac{q^9J_{11,121}}{J_{44,121}}S_8(1,11,1,2)-\frac{q^9J_{55,121}J_{121}^3}{J_{33,121}^2}[X(4)+1],\\
&S_5(1,11,1,2)=-\frac{q^9J_{11,121}}{J_{44,121}}S_8(1,11,1,2)+\frac{q^{20}J_{22,121}J_{121}^3}{J_{44,121}J_{55,121}}[X(3)-X(5)-X(2)-1],\\
&S_6(1,11,1,2)=\frac{q^5J_{22,121}}{J_{44,121}}S_8(1,11,1,2)-\frac{q^{27}J_{11,121}J_{22,121}J_{121}^3}{J_{33,121}J_{44,121}J_{55,121}}[X(3)+X(5)-X(1)-1],\\
&S_9(1,11,1,2)=-\frac{q^{-1}J_{55,121}}{J_{44,121}}S_8(1,11,1,2)+\frac{q^{21}J_{11,121}J_{121}^3}{J_{33,121}J_{44,121}}[2X(2)+X(3)],\\
&S_{10}(1,11,1,2)=\frac{q^{-1}J_{55,121}}{J_{44,121}}S_8(1,11,1,2)-\frac{q^{10}J_{22,121}J_{121}^3}{J_{11,121}J_{44,121}}[2X(3)+X(1)].
\end{align*}
Through \eqref{L-3-5}, we obtain
\begin{align*}
&S_0(1,11,1,2)+S_8(1,11,1,2)\\
=&-\sum_{n=1}^{\infty}\frac{q^{121n}}{(1-q^{121n})^2}-\frac{\mathcal{S}_1(q^{-77},q^{88};q^{121})[\mathcal{S}_1(q^{-77},q^{88};q^{121})+2]}{2}-\frac{\mathcal{S}_2(q^{-77},q^{88};q^{121})}{2}\\
=&-\sum_{n=1}^{\infty}\frac{q^{121n}}{(1-q^{121n})^2}-\frac{[X(3)+X(4)]^2-1}{2}+\frac{q^{33}J_{11,121}J_{121}^6}{2J_{44,121}J_{33,121}^2}+\frac{X(3)+X(4)-1}{2}.
\end{align*}
Adding these equations and simplifying, we arrive at \eqref{cor-2-2}. Similarly to the proof of \eqref{cor-2-2}, the same methodological approach can be applied to readily prove the remaining five equations.
\end{proof}

\begin{proof}[Proof of Theorem \ref{main-thm}]
To establish the proof of Theorem \ref{main-thm}, we first express $\mathcal{F}(b)$ in terms of eta-quotients, as presented in the Appendix. Initially, we consider the case where $b=1$. Utilizing \eqref{cor-1-1}-\eqref{cor-2-2}, we reformulate \eqref{generating-2} for $b=1$ as follows.
\begin{align}\label{mathcalF1}
&\mathcal{F}(1)=\sum_{n=0}^{\infty}(M_{\omega}(1,11,n)-M_{\omega}(10,11,n))q^n \\
=&\frac{11}{J_1}\sum_{m=0}^{10}(S_m(0,11,1,2)-S_m(1,11,1,2))-\frac{10}{J_1}\sum_{m=0}^{10}(S_m(0,11,1,1)-S_m(1,11,1,1))\nonumber\\
=&\frac{11}{J_1}
\Big \{
-q\frac{J_{22,121}J_{121}^3}{J_{11,121}^2}[X(5)+1]-q^{12}\frac{J_{33,121}J_{121}^3}{J_{22,121}J_{55,121}}[X(1)+X(2)+X(3)]\nonumber\\ &-q^3\frac{J_{33,121}J_{44,121}J_{121}^3}{J_{11,121}J_{22,121}J_{55,121}}[X(1)-X(2)-X(4)-1]+q^{14}\frac{J_{22,121}^2J_{121}^3}{J_{11,121}J_{33,121}J_{55,121}}[X(1)+X(2)\nonumber\\
&+X(3)]-q^{15}\frac{J_{121}^3}{J_{55,121}}[X(5)+1]+q^{6}\frac{J_{44,121}J_{121}^3}{J_{11,121}J_{55,121}}[X(1)-2X(3)-1]\nonumber\\
&-q^{17}\frac{J_{33,121}J_{121}^3}{J_{44,121}J_{55,121}}[2X(1)+X(4)]+q^{10}\frac{J_{22,121}J_{121}^3}{J_{11,121}J_{44,121}}[-X(1)+X(2)+X(4)+1]\nonumber\\ &+q^{21}\frac{J_{44,121}J_{121}^6}{J_{11,121}J_{55,121}^2}-\frac{[X(1)+X(5)]^2}{2}-\frac{q^{11}J_{44,121}J_{121}^6}{2J_{11,121}^2J_{55,121}^2}+\frac{X(1)+X(5)}{2}+q^{35}\frac{J_{11,121}J_{121}^6}{J_{33,121}J_{44,121}^2}\nonumber\\
&-q^2\frac{J_{55,121}J_{121}^3}{J_{11,121}J_{33,121}}[X(5)+X(3)-X(1)-1]-q^{14}\frac{J_{121}^3}{J_{44,121}}[X(4)+1]+\frac{q^{33}J_{11,121}J_{121}^6}{2J_{33,121}^2J_{44,121}}\nonumber\\
&+q^5\frac{J_{55,121}^2J_{121}^3}{J_{22,121}J_{33,121}J_{44,121}}[X(3)-X(2)-X(5)-1]+q^{27}\frac{J_{11,121}J_{22,121}J_{121}^3}{J_{33,121}J_{44,121}J_{55,121}}[X(3)+X(5)-X(1)\nonumber\\
&-1]+q^{9}\frac{J_{55,121}J_{121}^3}{J_{33,121}^2}[X(4)+1]-q^{20}\frac{J_{22,121}J_{121}^3}{J_{44,121}J_{55,121}}[X(3)-X(2)-X(5)-1]-\frac{X(3)+X(4)}{2}\nonumber\\ &-q^{21}\frac{J_{11,121}J_{121}^3}{J_{33,121}J_{44,121}}[2X(2)+X(3)]+q^{10}\frac{J_{22,121}J_{121}^3}{J_{11,121}J_{44,121}}[2X(3)+X(1)]+\frac{[X(3)+X(4)]^2}{2}
\Big \}\nonumber\\
&-\frac{10}{J_1}
\Big \{
-q\frac{J_{44,121}}{J_{55,121}}[2X(4)+X(1)+X(2)+1]+q^3\frac{J_{33,121}}{J_{55,121}}[3X(3)+X(2)+1]-q^{15}\frac{J_{121}^3}{J_{55,121}}\nonumber \\
&-q^6\frac{J_{22,121}}{J_{55,121}}[2X(2)+X(3)+X(4)+1]+q^{10}\frac{J_{11,121}}{J_{55,121}}[2X(1)+X(4)+X(5)+1]+X(1) \nonumber\\
&+X(5)+q^2\frac{J_{33,121}}{J_{44,121}}[2X(3)+X(1)+X(4)+1]-q^{14}\frac{J_{121}^3}{J_{44,121}}-q^5\frac{J_{22,121}}{J_{44,121}}[3X(2)+X(5)+1] \nonumber\\
&+q^9\frac{J_{11,121}}{J_{44,121}}[2X(1)+X(3)-X(5)+1]+q^{-1}\frac{J_{55,121}}{J_{44,121}}[2X(5)+X(2)-X(1)]-X(3)-X(4)
\Big \}.\nonumber
\end{align}

Drawing on the theory of modular functions and inspired by the method outlined in \cite{Chen-maple-2020}, we identify \eqref{mathcalF1} comprises two distinct  parts, denoted as $h_{\frac{1}{2}}(1)$ and $h_{\frac{3}{2}}(1)$.  For one thing,
\begin{align}\label{f-1-2}
h_{\frac{1}{2}}(1)
=&\frac{11}{J_1}
\Big \{
-q\frac{21J_{22,121}J_{121}^3}{22J_{11,121}^2}+q^{12}\frac{21J_{33,121}J_{121}^3}{22J_{22,121}J_{55,121}}+q^3\frac{21J_{33,121}J_{44,121}J_{121}^3}{22J_{11,121}J_{22,121}J_{55,121}}\\
&-q^{14}\frac{21J_{22,121}^2J_{121}^3}{22J_{11,121}J_{33,121}J_{55,121}}-q^{15}\frac{21J_{121}^3}{22J_{55,121}}-q^{6}\frac{21J_{44,121}J_{121}^3}{22J_{11,121}J_{55,121}}\nonumber\\
&+q^{17}\frac{21J_{33,121}J_{121}^3}{22J_{44,121}J_{55,121}}+q^{10}\frac{21J_{22,121}J_{121}^3}{22J_{11,121}J_{44,121}}+\frac{21\left[X(1)+X(5)+\frac{5}{11}\right]}{22}\nonumber\\
&+q^2\frac{19J_{55,121}J_{121}^3}{22J_{11,121}J_{33,121}}-q^{14}\frac{19J_{121}^3}{22J_{44,121}}-q^5\frac{19J_{55,121}^2J_{121}^3}{22J_{22,121}J_{33,121}J_{44,121}}\nonumber\\
&-q^{27}\frac{19J_{11,121}J_{22,121}J_{121}^3}{22J_{33,121}J_{44,121}J_{55,121}}+q^{9}\frac{19J_{55,121}J_{121}^3}{22J_{33,121}^2}+q^{20}\frac{19J_{22,121}J_{121}^3}{22J_{44,121}J_{55,121}}+q^{21}\frac{19J_{11,121}J_{121}^3}{22J_{33,121}J_{44,121}}\nonumber\\  &-q^{10}\frac{19J_{22,121}J_{121}^3}{22J_{11,121}J_{44,121}}-\frac{19\left[X(3)+X(4)+\frac{4}{11}\right]}{22}
\Big \}\nonumber\\
&-\frac{10}{J_1}
\Big \{
-q\frac{J_{44,121}}{J_{55,121}}[2X(4)+X(1)+X(2)+1]+q^3\frac{J_{33,121}}{J_{55,121}}[3X(3)+X(2)+1] \nonumber\\
&-q^{15}\frac{J_{121}^3}{J_{55,121}}-q^6\frac{J_{22,121}}{J_{55,121}}[2X(2)+X(3)+X(4)+1]+q^{10}\frac{J_{11,121}}{J_{55,121}}[2X(1)+X(4) \nonumber\\
&+X(5)+1]+q^2\frac{J_{33,121}}{J_{44,121}}[2X(3)+X(1)+X(4)+1]-q^{14}\frac{J_{121}^3}{J_{44,121}} \nonumber\\
&-q^5\frac{J_{22,121}}{J_{44,121}}[3X(2)+X(5)+1]+q^9\frac{J_{11,121}}{J_{44,121}}[2X(1)+X(3)-X(5)+1]\nonumber\\
&+q^{-1}\frac{J_{55,121}}{J_{44,121}}[2X(5)+X(2)-X(1)]+X(1)+X(5)-X(3)-X(4)+\frac{1}{11} \Big \}.\nonumber
\end{align}
We now prove $h_{\frac{1}{2}}(1)=g_{\frac{1}{2}}(1)$; where
\begin{align}\label{g-1-2}
g_{\frac{1}{2}}(1):=&\frac{9}{22}Y_0(q^{11})-\frac{q}{11}Y_1(q^{11})-\frac{2q^2}{11}Y_2(q^{11})+\frac{5q^3}{22}Y_3(q^{11})+\frac{q^4}{22}Y_4(q^{11})+\frac{4q^5}{11}Y_5(q^{11})
\\&+\frac{3q^7}{22}Y_7(q^{11})-\frac{q^8}{22}Y_8(q^{11})-\frac{5q^9}{22}Y_9(q^{11})+\frac{2q^{10}}{11}Y_{10}(q^{11}),\nonumber
\end{align}
and $Y_m(q)$ are defined in the Appendix for $0\leq m\leq 10$. Utilizing the algorithm presented in \cite{Chen-maple-2020}, we initially apply \eqref{f-1-2} to rephrase \eqref{g-1-2} as modular function identities for generalized eta-products on $\Gamma_1(121)$, conforming to the structure of \eqref{f-tau-1}. Then, using \cite[Theorem 2.9]{Chen-maple-2020}, we verify that each generalized eta-product is a modular function on $\Gamma_1(121)$. With the help of Theorems \ref{Thm-2.10} and \ref{Thm-2.11}, we compute the order of each generalized eta-product at each cusp of $\Gamma_1(121)$. Upon calculating the constant in \eqref{B}, we determine that $B=-946$. We define $g(\tau)$ as the expression on the right side of \eqref{f-tau-1} and demonstrate that Ord$(g(\tau),\infty,\Gamma_1(121))\textgreater 946$, which is a straightforward process. The desired identity is subsequently derived from Theorem \ref{Thm-2.12}. On the other hand, we have
\begin{align}\label{f-3-2}
h_{\frac{3}{2}}(1)
=&\frac{11}{J_1}
\Bigg \{
-q\frac{J_{22,121}J_{121}^3}{J_{11,121}^2}\left[X(5)+\frac{1}{22}\right]-q^{12}\frac{J_{33,121}J_{121}^3}{J_{22,121}J_{55,121}}\left[X(1)+X(2)+X(3)+\frac{21}{22}\right]\\ &-q^3\frac{J_{33,121}J_{44,121}J_{121}^3}{J_{11,121}J_{22,121}J_{55,121}}\left[X(1)-X(2)-X(4)-\frac{1}{22}\right]\nonumber\\
&+q^{14}\frac{J_{22,121}^2J_{121}^3}{J_{11,121}J_{33,121}J_{55,121}}\left[X(1)+X(2)+X(3)+\frac{21}{22}\right]\nonumber\\
&-q^{15}\frac{J_{121}^3}{J_{55,121}}\left[X(5)+\frac{1}{22}\right]+q^{6}\frac{J_{44,121}J_{121}^3}{J_{11,121}J_{55,121}}\left[X(1)-2X(3)-\frac{1}{22}\right]\nonumber\\
&-q^{17}\frac{J_{33,121}J_{121}^3}{J_{44,121}J_{55,121}}\left[2X(1)+X(4)+\frac{21}{22}\right]+q^{10}\frac{J_{22,121}J_{121}^3}{J_{11,121}J_{44,121}}\left[-X(1)+X(2)+X(4)+\frac{1}{22}\right]\nonumber\\ &+q^{21}\frac{J_{44,121}J_{121}^6}{J_{11,121}J_{55,121}^2}-\frac{\left[X(1)+X(5)+\frac{5}{11}\right]^2}{2}-\frac{q^{11}J_{44,121}J_{121}^6}{2J_{11,121}^2J_{55,121}}\nonumber\\
&-q^2\frac{J_{55,121}J_{121}^3}{J_{11,121}J_{33,121}}\left[X(5)+X(3)-X(1)-\frac{3}{22}\right]+q^{35}\frac{J_{11,121}J_{121}^6}{J_{33,121}J_{44,121}^2}\nonumber\\
&-q^{14}\frac{J_{121}^3}{J_{44,121}}\left[X(4)+\frac{3}{22}\right]+q^5\frac{J_{55,121}^2J_{121}^3}{J_{22,121}J_{33,121}J_{44,121}}\left[X(3)-X(2)-X(5)-\frac{3}{22}\right]\nonumber\\
&+q^{27}\frac{J_{11,121}J_{22,121}J_{121}^3}{J_{33,121}J_{44,121}J_{55,121}}\left[X(3)+X(5)-X(1)-\frac{3}{22}\right]+q^{9}\frac{J_{55,121}J_{121}^3}{J_{33,121}^2}\left[X(4)+\frac{3}{22}\right]\nonumber\\
&-q^{20}\frac{J_{22,121}J_{121}^3}{J_{44,121}J_{55,121}}\left[X(3)-X(2)-X(5)-\frac{3}{22}\right]-q^{21}\frac{J_{11,121}J_{121}^3}{J_{33,121}J_{44,121}}\left[2X(2)+X(3)+\frac{19}{22}\right]\nonumber\\ &+q^{10}\frac{J_{22,121}J_{121}^3}{J_{11,121}J_{44,121}}\left[2X(3)+X(1)+\frac{19}{22}\right]+\frac{\left[X(3)+X(4)+\frac{4}{11}\right]^2}{2}+\frac{q^{33}J_{11,121}J_{121}^6}{2J_{33,121}^2J_{44,121}}
\Bigg \}.\nonumber
\end{align}
Next we require to prove $h_{\frac{3}{2}}(1)=g_{\frac{3}{2}}(1)$; where
\begin{align}\label{g-3-2}
g_{\frac{3}{2}}(1):=&\left\{-\frac{9}{22}V_0(q^{11})t(q^{11})^{-1}-\frac{61}{22}V_0(q^{11})-\frac{95}{22}V_0(q^{11})t(q^{11})-\frac{49}{22}V_0(q^{11})T(q^{11})\right.\\\nonumber
&\left.+\frac{48}{11}V_0(q^{11})T(q^{11})t(q^{11})-\frac{95}{22}V_0(q^{11})T(q^{11})t(q^{11})^2\right\}+q\left\{-\frac{46}{11}V_1(q^{11})-\frac{42}{11}V_1(q^{11})t(q^{11})\right.\\\nonumber
&+\frac{4}{11}V_1(q^{11})t(q^{11})^2+\frac{36}{11}V_1(q^{11})T(q^{11})-\frac{71}{11}V_1(q^{11})T(q^{11})t(q^{11})-3V_1(q^{11})T(q^{11})t(q^{11})^3\\\nonumber
&\left.-\frac{145}{11}V_1(q^{11})T(q^{11})t(q^{11})^2\right\}+q^2\left\{\frac{2}{11}V_2(q^{11})t(q^{11})^{-1}-\frac{17}{11}V_2(q^{11})+\frac{15}{11}V_2(q^{11})t(q^{11})\right.\\\nonumber
&\left.+\frac{39}{11}V_2(q^{11})T(q^{11})-\frac{58}{11}V_2(q^{11})T(q^{11})t(q^{11})+\frac{4}{11}V_2(q^{11})T(q^{11})t(q^{11})^2\right\}\\\nonumber
&+q^3\left\{\frac{5}{11}V_3(q^{11})t(q^{11})^{-1}-\frac{153}{22}V_3(q^{11})-\frac{53}{22}V_3(q^{11})t(q^{11})-\frac{15}{22}V_3(q^{11})T(q^{11})t(q^{11})^{-1}\right.\\\nonumber
&\left.+\frac{73}{11}V_3(q^{11})T(q^{11})+\frac{60}{11}V_3(q^{11})T(q^{11})t(q^{11})\right\}+q^4\left\{\frac{122}{11}V_4(q^{11})t(q^{11})^{-1}+\frac{405}{22}V_4(q^{11})\right.\\\nonumber
&+\frac{161}{22}V_4(q^{11})t(q^{11})-\frac{245}{22}V_4(q^{11})T(q^{11})t(q^{11})^{-1}-\frac{135}{11}V_4(q^{11})T(q^{11})\\\nonumber
&\left.-\frac{10}{11}V_4(q^{11})T(q^{11})t(q^{11})\right\}+q^5\left\{-\frac{15}{11}V_5(q^{11})+\frac{103}{11}V_5(q^{11})t(q^{11})+9V_5(q^{11})t(q^{11})^2\right.\\\nonumber
&+\frac{16}{11}V_5(q^{11})t(q^{11})^3-7V_5(q^{11})T(q^{11})t(q^{11})-\frac{91}{11}V_5(q^{11})T(q^{11})t(q^{11})^2\\\nonumber
&\left.-\frac{16}{11}V_5(q^{11})T(q^{11})t(q^{11})^3\right\}+q^6\left\{-10V_6(q^{11})-19V_6(q^{11})t(q^{11})-6V_6(q^{11})t(q^{11})^2\right.\\\nonumber
&\left.+9V_6(q^{11})T(q^{11})+6V_6(q^{11})T(q^{11})t(q^{11})+6V_6(q^{11})T(q^{11})t(q^{11})^2\right\}\\\nonumber
&+q^7\left\{\frac{141}{22}V_7(q^{11})t(q^{11})^{-1}+\frac{56}{11}V_7(q^{11})+\frac{197}{22}V_7(q^{11})t(q^{11})+\frac{21}{22}V_7(q^{11})t(q^{11})^2\right.\\\nonumber
&-\frac{141}{22}V_7(q^{11})T(q^{11})t(q^{11})^{-1}-\frac{9}{11}V_7(q^{11})T(q^{11})-8V_7(q^{11})T(q^{11})t(q^{11})\\\nonumber
&\left.-\frac{21}{22}V_7(q^{11})T(q^{11})t(q^{11})^2\right\}+q^8\left\{-2V_8(q^{11})t(q^{11})^{-1}-\frac{21}{22}V_8(q^{11})+\frac{107}{22}V_8(q^{11})t(q^{11})\right.\\\nonumber
&+\frac{79}{22}V_8(q^{11})t(q^{11})^2-\frac{107}{22}V_8(q^{11})T(q^{11})t(q^{11})+\frac{2}{11}V_8(q^{11})T(q^{11})t(q^{11})^2+\\\nonumber
&\left.\frac{79}{22}V_8(q^{11})T(q^{11})t(q^{11})^3\right\}+q^9\left\{\frac{15}{22}V_9(q^{11})t(q^{11})^{-2}+\frac{4}{11}V_9(q^{11})t(q^{11})^{-1}\right.\\\nonumber
&-\frac{5}{11}V_9(q^{11})-\frac{15}{22}V_9(q^{11})T(q^{11})t(q^{11})^{-2}+\frac{7}{22}V_9(q^{11})T(q^{11})t(q^{11})^{-1}-\frac{29}{11}V_9(q^{11})T(q^{11})\\\nonumber
&\left.-\frac{1}{22}V_9(q^{11})T(q^{11})t(q^{11})\right\}+q^{10}\left\{\frac{75}{11}V_{10}(q^{11})+\frac{39}{11}V_{10}(q^{11})t(q^{11})-\frac{4}{11}V_{10}(q^{11})t(q^{11})^2\right.\\\nonumber
&\left.+\frac{16}{11}V_{10}(q^{11})T(q^{11})t(q^{11})+\frac{14}{11}V_{10}(q^{11})T(q^{11})t(q^{11})^2-\frac{4}{11}V_{10}(q^{11})T(q^{11})t(q^{11})^3\right\},
\end{align}
and $T(q)$, $t(q)$, and $V_m(q)$ for $0\leq m\leq 10$ are defined in the appendix. By employing the same method with Garvan's MAPLE program, we calculate that $B=-1683$ and can readily demonstrate its equivalence. It is easy to delineate the 11-dissection of $g_{\frac{1}{2}}(1)$ and $g_{\frac{3}{2}}(1)$, which leads us to the following result:
\begin{align}\label{6-1}
f_6(1)=-10V_6-19V_6t-6V_6t^2+9V_6T+6V_6Tt+6V_6Tt^2;
\end{align}
where $V_6$, $T$, and $t$ are defined in the  Appendix. Similarly, for $b=2,3,4,5$, we can calculate other eight formula's $B$ using Garvan's MAPLE program:
\begin{align*}
&B_{\frac{1}{2},2}=-935,&B_{\frac{3}{2},2}=-1683;\\
&B_{\frac{1}{2},3}=-946,&B_{\frac{3}{2},3}=-1683;\\
&B_{\frac{1}{2},4}=-935,&B_{\frac{3}{2},4}=-1683;\\
&B_{\frac{1}{2},5}=-814,&B_{\frac{3}{2},5}=-1683;
\end{align*}
where $B_{l,b}$ denotes the value of $B$ within the formula corresponding to parameter $b$ and weight $l$.  Next, We  prove eight identities and calculate their 11-dissection, which leads to the  derivation of the following identities:
\begin{align}
&f_6(2)=2V_6+17V_6t+10V_6t^2-4V_6T-10V_6Tt-10V_6Tt^2,\label{6-2}\\
&f_6(3)=3V_6-13V_6t-7V_6t^2-6V_6T+7V_6Tt+7V_6Tt^2,\label{6-3}\\
&f_6(4)=-7V_6+V_6t-2V_6t^2+3V_6T+2V_6Tt+2V_6Tt^2,\label{6-4}\\
&f_6(5)=5V_6+4V_6t+3V_6t^2+V_6T-3V_6Tt-3V_6Tt^2.\label{6-5}
\end{align}
Furthermore, we have determined the 11-dissection of the generating functions for $M_{\omega}(b,11,n)-M_{\omega}(11-b,11,n)$, as detailed in the appendix. Combining \eqref{6-1}-\eqref{6-5}, we arrive at
$$
f_6(1)+2f_6(2)+3f_6(3)+4f_6(4)+5f_6(5)=0.
$$
We  thus complete the proof of Theorem \ref{main-thm}.
\end{proof}

\section*{Appendix}
Define
\begin{align*}
&T:=T(q)=\frac{J_{2,11}^2J_{5,11}}{J_{1,11}J_{4,11}^2},\\
&t:=t(q)=\frac{qJ_{1,11}J_{4,11}}{J_{5,11}^2},\\
&f_m(b)=\sum_{n=0}^{\infty}(M_{\omega}(b,11,11n+m)-M_{\omega}(11-b,11,11n+m))q^n.\\
\end{align*}
We have
\begin{itemize}
\item $m=0$ \\
\begin{align*}
&f_0(1)=-\frac{9}{22}V_0t^{-1}-\frac{61}{22}V_0-\frac{95}{22}V_0t-\frac{49}{22}V_0T+\frac{48}{11}V_0Tt-\frac{95}{22}V_0Tt^2+\frac{9}{22}Y_0,\\
&f_0(2)=-\frac{7}{22}V_0t^{-1}+\frac{87}{22}V_0-\frac{47}{22}V_0t-\frac{65}{22}V_0T-\frac{3}{11}V_0Tt-\frac{47}{22}V_0Tt^2+\frac{7}{22}Y_0,\\
&f_0(3)=-\frac{5}{22}V_0t^{-1}-\frac{249}{22}V_0+\frac{1}{22}V_0t+\frac{161}{22}V_0T-\frac{54}{11}V_0Tt+\frac{1}{22}V_0Tt^2+\frac{5}{22}Y_0,\\
&f_0(4)=-\frac{3}{22}V_0t^{-1}+\frac{141}{22}V_0+\frac{49}{22}V_0t-\frac{97}{22}V_0T+\frac{16}{11}V_0Tt+\frac{49}{22}V_0Tt^2+\frac{3}{22}Y_0,\\
&f_0(5)=-\frac{1}{22}V_0t^{-1}+\frac{47}{22}V_0+\frac{97}{22}V_0t-\frac{113}{22}V_0T-\frac{35}{11}V_0Tt+\frac{97}{22}V_0Tt^2+\frac{1}{22}Y_0,\\
&\text{where}~~ V_0:=V_0(q)=\frac{qJ_{4,11}J_{11}^5}{J_{2,11}^2J_{3,11}}, Y_0:=Y_0(q)=\frac{J_{11}^2}{J_{1,11}}.
\end{align*}
\item $m=1$ \\
\begin{align*}
&f_1(1)=-\frac{46}{11}V_1-\frac{42}{11}V_1t+\frac{4}{11}V_1t^2+\frac{36}{11}V_1T-\frac{71}{11}V_1Tt-\frac{145}{11}V_1Tt^2-3V_1Tt^3-\frac{1}{11}Y_1,\\
&f_1(2)=-\frac{4}{11}V_1-\frac{223}{22}V_1t-\frac{215}{22}V_1t^2+\frac{1}{22}V_1T+\frac{78}{11}V_1Tt+\frac{40}{11}V_1Tt^2-6V_1Tt^3+\frac{7}{22}Y_1,\\
&f_1(3)=\frac{49}{11}V_1+\frac{221}{22}V_1t+\frac{123}{22}V_1t^2-\frac{103}{22}V_1T-\frac{48}{11}V_1Tt-\frac{6}{11}V_1Tt^2+2V_1Tt^3+\frac{5}{22}Y_1,\\
&f_1(4)=-\frac{19}{11}V_1-\frac{61}{22}V_1t-\frac{23}{22}V_1t^2+\frac{35}{22}V_1T-\frac{53}{11}V_1Tt-\frac{52}{11}V_1Tt^2-V_1Tt^3+\frac{3}{22}Y_1,\\
&f_1(5)=-\frac{87}{11}V_1-\frac{101}{22}V_1t+\frac{73}{22}V_1t^2+\frac{173}{22}V_1T+\frac{63}{11}V_1Tt+\frac{23}{11}V_1Tt^2+7V_1Tt^3+\frac{1}{22}Y_1,\\
&\text{where}~~ V_1:=V_1(q)=\frac{J_{5,11}^2J_{11}^5}{J_{2,11}^4}, Y_1:=Y_1(q)=\frac{J_{5,11}J_{11}^2}{J_{2,11}J_{3,11}}.
\end{align*}
\item $m=2$ \\
\begin{align*}
&f_2(1)=\frac{2}{11}V_2t^{-1}-\frac{17}{11}V_2+\frac{15}{11}V_2t+\frac{39}{11}V_2T-\frac{58}{11}V_2Tt+\frac{4}{11}V_2Tt^2-\frac{2}{11}Y_2,\\
&f_2(2)=-\frac{47}{22}V_2t^{-1}-\frac{145}{22}V_2-\frac{171}{22}V_2t-\frac{9}{22}V_2T+\frac{16}{11}V_2Tt+\frac{27}{22}V_2Tt^2+\frac{3}{22}Y_2,\\
&f_2(3)=-\frac{5}{11}V_2t^{-1}-\frac{18}{11}V_2+\frac{23}{11}V_2t-\frac{37}{11}V_2T+\frac{24}{11}V_2Tt-\frac{10}{11}V_2Tt^2+\frac{5}{11}Y_2,\\
&f_2(4)=-\frac{3}{11}V_2t^{-1}+\frac{86}{11}V_2+\frac{38}{11}V_2t+\frac{2}{11}V_2T-\frac{34}{11}V_2Tt-\frac{6}{11}V_2Tt^2+\frac{3}{11}Y_2,\\
&f_2(5)=-\frac{1}{11}V_2t^{-1}-\frac{173}{11}V_2-\frac{68}{11}V_2t+\frac{41}{11}V_2T+\frac{29}{11}V_2Tt-\frac{2}{11}V_2Tt^2+\frac{1}{11}Y_2,\\
&\text{where}~~ V_2:=V_2(q)=\frac{qJ_{11}^5}{J_{2,11}^2}, Y_2:=Y_2(q)=\frac{J_{3,11}J_{11}^2}{J_{1,11}J_{4,11}}.
\end{align*}
\item $m=3$ \\
\begin{align*}
&f_3(1)=\frac{5}{11}V_3t^{-1}-\frac{153}{22}V_3-\frac{53}{22}V_3t-\frac{15}{22}V_3Tt^{-1}+\frac{73}{11}V_3T+\frac{60}{11}V_3Tt+\frac{5}{22}Y_3,\\
&f_3(2)=-\frac{1}{11}V_3t^{-1}+\frac{79}{22}V_3+\frac{59}{22}V_3t+\frac{3}{22}V_3Tt^{-1}-\frac{63}{11}V_3T-\frac{12}{11}V_3Tt-\frac{1}{22}Y_3,\\
&f_3(3)=-\frac{51}{11}V_3t^{-1}+\frac{18}{11}V_3-\frac{8}{11}V_3t+\frac{16}{11}V_3Tt^{-1}-\frac{67}{11}V_3T-\frac{7}{11}V_3Tt+\frac{2}{11}Y_3,\\
&f_3(4)=\frac{9}{11}V_3t^{-1}-\frac{227}{22}V_3-\frac{47}{22}V_3t-\frac{27}{22}V_3Tt^{-1}+\frac{83}{11}V_3T-\frac{13}{11}V_3Tt+\frac{9}{22}Y_3,\\
&f_3(5)=\frac{3}{11}V_3t^{-1}+\frac{247}{22}V_3+\frac{65}{22}V_3t-\frac{9}{22}V_3Tt^{-1}-\frac{53}{11}V_3T-\frac{85}{11}V_3Tt+\frac{3}{22}Y_3,\\
&\text{where}~~ V_3:=V_3(q)=\frac{qJ_{4,11}J_{11}^5}{J_{2,11}^2J_{5,11}}, Y_3:=Y_3(q)=\frac{J_{2,11}J_{11}^2}{J_{1,11}J_{3,11}}.
\end{align*}
\item $m=4$ \\
\begin{align*}
&f_4(1)=\frac{122}{11}V_4t^{-1}+\frac{405}{22}V_4+\frac{161}{22}V_4t-\frac{245}{22}V_4Tt^{-1}-\frac{135}{11}V_4T-\frac{10}{11}V_4Tt+\frac{1}{22}Y_4,\\
&f_4(2)=-\frac{75}{11}V_4t^{-1}-\frac{123}{11}V_4-\frac{48}{11}V_4t+\frac{52}{11}V_4Tt^{-1}+\frac{82}{11}V_4T+\frac{24}{11}V_4Tt+\frac{1}{11}Y_4,\\
&f_4(3)=\frac{3}{11}V_4t^{-1}+\frac{5}{22}V_4-\frac{1}{22}V_4t-\frac{9}{22}V_4Tt^{-1}-\frac{42}{11}V_4T-\frac{30}{11}V_4Tt+\frac{3}{22}Y_4,\\
&f_4(4)=-\frac{29}{11}V_4t^{-1}-\frac{4}{11}V_4+\frac{25}{11}V_4t-\frac{17}{11}V_4Tt^{-1}-\frac{78}{11}V_4T-\frac{73}{11}V_4Tt+\frac{2}{11}Y_4,\\
&f_4(5)=\frac{5}{11}V_4t^{-1}-\frac{153}{22}V_4-\frac{163}{22}V_4t-\frac{15}{22}V_4Tt^{-1}+\frac{51}{11}V_4T+\frac{71}{11}V_4Tt+\frac{5}{22}Y_4,\\
&\text{where}~~ V_4:=V_4(q)=\frac{qJ_{1,11}J_{3,11}J_{4,11}J_{11}^5}{J_{2,11}^4J_{5,11}}, Y_4:=Y_4(q)=\frac{J_{11}^2}{J_{2,11}}.
\end{align*}
\item $m=5$ \\
\begin{align*}
&f_5(1)=-\frac{15}{11}V_5+\frac{103}{11}V_5t+9V_5t^2+\frac{16}{11}V_5t^3-7V_5Tt-\frac{91}{11}V_5Tt^2-\frac{16}{11}V_5Tt^3+\frac{4}{11}Y_5,\\
&f_5(2)=-\frac{5}{22}V_5-\frac{369}{22}V_5t-15V_5t^2-\frac{35}{22}V_5t^3+8V_5Tt+\frac{126}{11}V_5Tt^2+\frac{35}{22}V_5Tt^3+\frac{5}{22}Y_5,\\
&f_5(3)=-\frac{34}{11}V_5+\frac{56}{11}V_5t+5V_5t^2+\frac{4}{11}V_5t^3+V_5Tt-\frac{53}{11}V_5Tt^2-\frac{4}{11}V_5Tt^3+\frac{1}{11}Y_5,\\
&f_5(4)=\frac{1}{22}V_5-\frac{23}{22}V_5t+3V_5t^2+\frac{7}{22}V_5t^3-6V_5Tt-\frac{1}{11}V_5Tt^2-\frac{7}{22}V_5Tt^3-\frac{1}{22}Y_5,\\
&f_5(5)=-\frac{53}{11}V_5+\frac{9}{11}V_5t+V_5t^2-\frac{8}{11}V_5t^3-2V_5Tt-\frac{15}{11}V_5Tt^2+\frac{8}{11}V_5Tt^3-\frac{2}{11}Y_5,\\
&\text{where}~~ V_5:=V_5(q)=\frac{J_{3,11}J_{5,11}J_{11}^5}{J_{1,11}J_{2,11}^2J_{4,11}}, Y_5:=Y_5(q)=\frac{J_{4,11}J_{11}^2}{J_{2,11}J_{5,11}}.
\end{align*}
\item $m=6$ \\
\begin{align*}
&f_6(1)=-10V_6-19V_6t-6V_6t^2+9V_6T+6V_6Tt+6V_6Tt^2,\\
&f_6(2)=2V_6+17V_6t+10V_6t^2-4V_6T-10V_6Tt-10V_6Tt^2,\\
&f_6(3)=3V_6-13V_6t-7V_6t^2-6V_6T+7V_6Tt+7V_6Tt^2,\\
&f_6(4)=-7V_6+V_6t-2V_6t^2+3V_6T+2V_6Tt+2V_6Tt^2,\\
&f_6(5)=5V_6+4V_6t+3V_6t^2+V_6T-3V_6Tt-3V_6Tt^2,\\
&\text{where}~~ V_6:=V_6(q)=\frac{J_{3,11}J_{11}^5}{J_{1,11}J_{2,11}^2}.
\end{align*}
\item $m=7$ \\
\begin{align*}
&f_7(1)=\frac{141}{22}V_7t^{-1}+\frac{56}{11}V_7+\frac{197}{22}V_7t+\frac{21}{22}V_7t^2-\frac{141}{22}V_7Tt^{-1}-\frac{9}{11}V_7T-8V_7Tt-\frac{21}{22}V_7Tt^2+\frac{3}{22}Y_7,\\
&f_7(2)=-\frac{46}{11}V_7t^{-1}-\frac{64}{11}V_7-\frac{78}{11}V_7t-\frac{12}{11}V_7t^2+\frac{46}{11}V_7Tt^{-1}-\frac{7}{11}V_7T+6V_7Tt+\frac{12}{11}V_7Tt^2+\frac{3}{11}Y_7,\\
&f_7(3)=-\frac{25}{11}V_7t^{-1}-\frac{19}{11}V_7+\frac{26}{11}V_7t+\frac{4}{11}V_7t^2+\frac{25}{11}V_7Tt^{-1}-\frac{38}{11}V_7T-2V_7Tt-\frac{4}{11}V_7Tt^2-\frac{1}{11}Y_7,\\
&f_7(4)=\frac{47}{22}V_7t^{-1}+\frac{59}{11}V_7-\frac{15}{22}V_7t+\frac{7}{22}V_7t^2-\frac{47}{22}V_7Tt^{-1}-\frac{3}{11}V_7T+V_7Tt-\frac{7}{22}V_7Tt^2+\frac{1}{22}Y_7,\\
&f_7(5)=\frac{50}{11}V_7t^{-1}-\frac{83}{11}V_7-\frac{52}{11}V_7t-\frac{8}{11}V_7t^2-\frac{50}{11}V_7Tt^{-1}+\frac{76}{11}V_7T+4V_7Tt+\frac{8}{11}V_7Tt^2+\frac{2}{11}Y_7,\\
&\text{where}~~ V_7:=V_7(q)=\frac{J_{3,11}J_{4,11}J_{11}^5}{J_{1,11}J_{2,11}^2J_{5,11}}, Y_7:=Y_7(q)=\frac{J_{11}^2}{J_{3,11}}.
\end{align*}
\item $m=8$ \\
\begin{align*}
&f_8(1)=-2V_8t^{-1}-\frac{21}{22}V_8+\frac{107}{22}V_8t+\frac{79}{22}V_8t^2-\frac{107}{22}V_8Tt+\frac{2}{11}V_8Tt^2+\frac{79}{22}V_8Tt^3-\frac{1}{22}Y_8,\\
&f_8(2)=-4V_8t^{-1}-\frac{120}{11}V_8-\frac{80}{11}V_8t+\frac{2}{11}V_8t^2-\frac{41}{11}V_8Tt-\frac{29}{11}V_8Tt^2+\frac{2}{11}V_8Tt^3-\frac{1}{11}Y_8,\\
&f_8(3)=5V_8t^{-1}+\frac{179}{22}V_8+\frac{79}{22}V_8t-\frac{5}{22}V_8t^2+\frac{163}{22}V_8Tt+\frac{6}{11}V_8Tt^2-\frac{5}{22}V_8Tt^3-\frac{3}{22}Y_8,\\
&f_8(4)=-8V_8t^{-1}-\frac{95}{22}V_8-\frac{23}{22}V_8t-\frac{69}{22}V_8t^2-\frac{219}{22}V_8Tt-\frac{14}{11}V_8Tt^2-\frac{69}{22}V_8Tt^3+\frac{7}{22}Y_8,\\
&f_8(5)=V_8t^{-1}-\frac{124}{11}V_8-\frac{123}{11}V_8t-\frac{6}{11}V_8t^2+\frac{2}{11}V_8Tt-\frac{34}{11}V_8Tt^2-\frac{6}{11}V_8Tt^3+\frac{3}{11}Y_8,\\
&\text{where}~~ V_8:=V_8(q)=\frac{qJ_{1,11}J_{5,11}J_{11}^5}{J_{2,11}^2J_{3,11}J_{4,11}}, Y_8:=Y_8(q)=\frac{qJ_{1,11}J_{11}^2}{J_{4,11}J_{5,11}}.
\end{align*}
\item $m=9$ \\
\begin{align*}
&f_9(1)=\frac{15}{22}V_9t^{-2}+\frac{4}{11}V_9t^{-1}-\frac{5}{11}V_9-\frac{15}{22}V_9Tt^{-2}+\frac{7}{22}V_9Tt^{-1}-\frac{29}{11}V_9T-\frac{1}{22}V_9Tt-\frac{5}{22}Y_9,\\
&f_9(2)=-\frac{3}{22}V_9t^{-2}-\frac{25}{11}V_9t^{-1}+\frac{1}{11}V_9+\frac{3}{22}V_9Tt^{-2}+\frac{47}{22}V_9Tt^{-1}+\frac{30}{11}V_9T+\frac{97}{22}V_9Tt+\frac{1}{22}Y_9,\\
&f_9(3)=-\frac{21}{22}V_9t^{-2}+\frac{67}{11}V_9t^{-1}+\frac{7}{11}V_9+\frac{21}{22}V_9Tt^{-2}-\frac{155}{22}V_9Tt^{-1}-\frac{32}{11}V_9T-\frac{47}{22}V_9Tt+\frac{7}{22}Y_9,\\
&f_9(4)=\frac{8}{11}V_9t^{-2}-\frac{28}{11}V_9t^{-1}+\frac{35}{11}V_9-\frac{8}{11}V_9Tt^{-2}+\frac{36}{11}V_9Tt^{-1}-\frac{39}{11}V_9T-\frac{57}{11}V_9Tt+\frac{1}{11}Y_9,\\
&f_9(5)=\frac{9}{22}V_9t^{-2}-\frac{46}{11}V_9t^{-1}-\frac{124}{11}V_9-\frac{9}{22}V_9Tt^{-2}+\frac{101}{22}V_9Tt^{-1}+\frac{31}{11}V_9T-\frac{49}{22}V_9Tt-\frac{3}{22}Y_9,\\
&\text{where}~~ V_9:=V_9(q)=\frac{J_{4,11}^2J_{11}^5}{J_{2,11}^3J_{5,11}}, Y_9:=Y_9(q)=\frac{J_{11}^2}{J_{4,11}}.
\end{align*}
\item $m=10$ \\
\begin{align*}
&f_{10}(1)=\frac{75}{11}V_{10}+\frac{39}{11}V_{10}t-\frac{4}{11}V_{10}t^2+\frac{16}{11}V_{10}Tt+\frac{14}{11}V_{10}Tt^2-\frac{4}{11}V_{10}Tt^3+\frac{2}{11}Y_{10},\\
&f_{10}(2)=-\frac{92}{11}V_{10}-\frac{43}{11}V_{10}t-\frac{8}{11}V_{10}t^2+\frac{32}{11}V_{10}Tt+\frac{28}{11}V_{10}Tt^2-\frac{8}{11}V_{10}Tt^3+\frac{4}{11}Y_{10},\\
&f_{10}(3)=\frac{43}{22}V_{10}+\frac{3}{22}V_{10}t+\frac{9}{22}V_{10}t^2-\frac{157}{22}V_{10}Tt-\frac{46}{11}V_{10}Tt^2+\frac{9}{22}V_{10}Tt^3+\frac{1}{22}Y_{10},\\
&f_{10}(4)=-\frac{52}{11}V_{10}+\frac{2}{11}V_{10}t+\frac{6}{11}V_{10}t^2-\frac{24}{11}V_{10}Tt-\frac{21}{11}V_{10}Tt^2+\frac{6}{11}V_{10}Tt^3-\frac{3}{11}Y_{10},\\
&f_{10}(5)=\frac{23}{11}V_{10}-\frac{80}{11}V_{10}t+\frac{2}{11}V_{10}t^2+\frac{113}{11}V_{10}Tt-\frac{7}{11}V_{10}Tt^2+\frac{2}{11}V_{10}Tt^3-\frac{1}{11}Y_{10},\\
&\text{where}~~ V_{10}:=V_{10}(q)=\frac{J_{5,11}^2J_{11}^5}{J_{2,11}^2J_{3,11}}, Y_{10}:=Y_{10}(q)=\frac{J_{11}^2}{J_{5,11}}.
\end{align*}
\end{itemize}

\subsection*{Acknowledgements}

The first author was supported in part by  the National Natural Science Foundation of China (Grant No. 12201387).







\end{document}